\documentclass[final,reqno,onefignum,onetabnum]{siamltex1213}

\usepackage{subdepth}
\usepackage{overpic}
\usepackage{amsmath,amstext,amsbsy,amssymb,mathdots}
\usepackage{booktabs,bm}
\usepackage{mathrsfs}
\usepackage{tikz,todonotes}
\usepackage{tikz}
\usetikzlibrary{positioning}
\usetikzlibrary{shapes,arrows}
\usepackage[fleqn,tbtags]{mathtools}
\usepackage{relsize}
\usepackage{xcolor,colortbl}

\title{Gaussian elimination corrects pivoting mistakes}
\author{Alex Townsend\thanks{Department 
of Mathematics, Massachusetts Institute of Technology, 77 Massachusetts Avenue
Cambridge, MA 02139-4307. (\texttt{ajt@mit.edu})}}

\date{}
\begin{document}
\maketitle

\begin{abstract} 
Gaussian elimination (GE) is the archetypal direct algorithm for solving 
linear systems of equations and this has been its primary application for thousands of years.  
In the last decade, GE has found another major use as an iterative algorithm for low rank approximation. 
In this setting, GE is often employed with complete pivoting and designed to allow for non-optimal pivoting, i.e., pivoting mistakes, 
that could render GE numerically unstable when implemented in floating point arithmetic. While
it may appear that pivoting mistakes could accumulate and lead to a large growth factor, we show that later GE steps
correct earlier pivoting mistakes, even while more are being made. In short, GE is very robust 
to non-optimal pivots, allowing for its iterative variant to flourish. 
\end{abstract}



\section{Introduction} 
Gaussian elimination (GE) is the archetypal direct algorithm
for solving linear systems of equations~\cite{Golub_12_01}. 
For an $n\times n$ invertible matrix $A$, GE performs a total of $n$ steps to 
solve $Ax=b$, where each step requires a nonzero entry to be selected as a pivot. After 
$n$ GE steps on $A$, the solution to $Ax=b$ can be calculated by forward and back
substitution~\cite{Golub_12_01}.  This is the elimination procedure that any reader who 
has taken an introductory linear algebra course is familiar with.  

In the last decade, an iterative variant of GE has become popular for low rank approximation 
as a substitute for the computationally expensive singular value decomposition~\cite{Bebendorf_00_01,Pan_00_01,Townsend_13_01}. 
As an iterative algorithm, $k\leq n$ GE steps on $A$ are performed 
to calculate a rank $k$ approximation to $A$. In this setting it is beneficial if $A$ is low rank,
rectangular, or possesses rapidly decaying singular values as then GE will typically terminate 
after a handful of steps. 

Each step of GE in ``iterative mode" is mathematically equivalent to the familiar direct algorithm, though it is described slightly differently. 
Let $A$ be a nonzero $m\times n$ matrix and set $A^{(0)}=A$. First, GE selects a nonzero entry 
of $A^{(0)}$ called a {\em pivot}, say the $A_{i_1,j_1}^{(0)}$ entry, and uses row $i_1$ to 
eliminate column $j_1$, i.e., 
\begin{equation}
 A^{(1)} = A^{(0)} - \underbrace{A_{:,j_1}^{(0)}A_{i_1,:}^{(0)}\Big/A_{i_1,j_1}^{(0)}}_{\text{Rank $1$ matrix}},
\label{eq:GEFirstStep}
\end{equation}
where $A_{:,j_1}$ and $A_{i_1,:}$ denotes column $j_1$ and row $i_1$ of $A$, respectively.
Row $i_1$ and column $j_1$ of the matrix $A^{(1)}$ are zero, and  
the matrix $A_{:,j_1}^{(0)}A_{i_1,:}^{(0)}/A_{i_1,j_1}^{(0)}$ is of rank $1$ because
it is the outerproduct of a column and row vector.  Depending on the choice of the pivot, the matrix 
$A_{:,j_1}^{(0)}A_{i_1,:}^{(0)}\Big/A_{i_1,j_1}^{(0)}$ can be a near-best rank $1$ 
approximation to $A$~\cite{Foster_06_01,Townsend_13_01}. 

If $A^{(1)}$ is the zero matrix, then the GE procedure is terminated; otherwise, 
a nonzero entry of $A^{(1)}$ is selected, say the $A_{i_2,j_2}^{(1)}$ entry,
and a second GE step is performed, i.e., 
\begin{equation} 
  A^{(2)} = A^{(1)} - \underbrace{A_{:,j_2}^{(1)}A_{i_2,:}^{(1)}\Big/A_{i_2,j_2}^{(1)}}_{\text{Rank $1$ matrix}}. 
\label{eq:GESecondStep}
\end{equation}
Now, the matrix $A^{(2)}$ is zero in row $i_1$ and $i_2$ as well as column $j_1$ 
and $j_2$. The two rank $1$ matrices from~\eqref{eq:GEFirstStep} and~\eqref{eq:GESecondStep}
can be added together to form a rank $2$ approximation to $A$.  
The GE procedure is terminated if $A^{(2)}$ is the zero matrix; otherwise,
a nonzero entry of $A^{(2)}$ is selected followed by another GE step. 
In principle, GE continues for $k$ steps until the matrix 
$A^{(k)}$ is zero or its entries are considered to be sufficiently small in magnitude.

By the Wedderburn--Guttman Theorem~\cite{Guttman_44_01}, each GE step reduces the rank of $A$ by precisely 
one and hence, GE is guaranteed to terminate after at most $\min(m,n)$ steps. The rank $1$ matrices from the first $k$ GE steps can be
accumulated to form a rank $k$ approximation to $A$.  Other direct algorithms can also be used in an 
iterative manner to construct low rank approximations~\cite{Gu_96_01}. 

The strategy for selecting GE pivots is very important as it alters the numerical stability and computational 
efficiency of the algorithm as well as the near-optimality of the constructed low rank approximations~\cite{Foster_06_01}. 
There are many pivoting strategies such as (in order of computational cost)
partial pivoting (pivot is the absolute maximum entry in a column), rook pivoting
(pivot is the absolute maximum entry in its column and row), and 
complete pivoting (pivot is the absolute maximum entry in the matrix).  
Unfortunately, partial pivoting is not suitable for GE in iterative mode~\cite{Foster_06_01}, though it is 
the standard choice when solving linear systems of equations.
Instead, complete or rook pivoting is often employed~\cite{Bebendorf_00_01,Foster_06_01,Townsend_13_01}.  
We will focus on complete pivoting. 

Complete pivoting is prohibitively expensive because at each step the whole matrix must be searched to find the absolute 
maximum entry.  For an $m\times n$ matrix, this costs $\mathcal{O}(mn)$ operations and 
a full search must be repeated before each GE step.  In many applications this is too costly as $\min(m,n)$ is in the thousands, 
and the matrix entries may be computed on-demand as opposed to stored~\cite{Bebendorf_00_01}.  
Instead, it is common to only find an entry with a sufficiently large absolute value by searching a small proportion of the matrix.  
A partial search of the matrix may happen to find the absolute 
maximum entry of $A$ for the pivot, but often it will not. We regard such a 
pivoting strategy as complete pivoting with mistakes. These pivoting mistakes could render the
GE procedure numerically unstable and useless when performed in floating point arithmetic. 
Here, we show that later GE steps seem to correct earlier pivoting mistakes to allow the iterative variant 
of GE to be used numerically. It is currently being employed with various algorithmic details in: (1) hierarchical matrix 
compression, where it is called adaptive cross approximation~\cite{Bebendorf_00_01,Bebendorf_08_01}, (2) low rank function 
approximation, where it is sometimes refered to as Geddes--Newton approximation~\cite{Geddes_08_01,Townsend_13_02}, 
and (3) randomized techniques, where it computes a two-sided interpolative decomposition with relaxed 
constraints~\cite{Halko_11_01}. 

\section{Growth factors in Gaussian elimination}
An important quantity for the analysis of the GE procedure is the growth factor~\cite{Higham_89_01,Higham_02_01}, 
denoted by $\rho(A)$, which is defined as the maximum relative amplification of the matrix entries during the GE procedure. Since we 
are interested in the iterative analogue, we also define the term {\em intermediate growth factor}. 
\begin{definition}[Intermediate growth factor]
For an $m\times n$ matrix $A$, the intermediate growth factor $\rho_k(A)$ for GE step $k$ with $1\leq k\leq \min(m,n)$ is given by the 
ratio between the absolute maximum entry of $A^{(k)}$ and the original matrix $A$. That is, 
\[
\rho_k(A) = \max_{1\leq i\leq m}\max_{1\leq j\leq n} \left|A_{ij}^{(k)}\right| \Bigg/  \max_{1\leq i\leq m}\max_{1\leq j\leq n} \left|A_{ij}\right|.
\]
For complete pivoting, the growth factor is given by $\rho(A) = \max_{0\leq k\leq \min(m,n)-1} \rho_k(A)$. 
\end{definition} 

Given a pivoting strategy, if the growth factor $\rho(A)$ is very large then GE with that 
pivoting strategy is not a backward stable algorithm, see~\cite[Thm.~9.5]{Higham_02_01}.
The literature regarding the growth factor for different pivoting strategies is summarized in~\cite[Sec.~4.3]{Golub_12_01}. 
There are many interesting practical nuances regarding GE, its numerical stability, and 
the growth factor. See, for example,~\cite{Driscoll_07_01,Trefethen_90_01}.


For GE with complete pivoting, Wilkinson showed that $\rho(A)\leq 2\sqrt{n} n^{\tfrac{1}{4}\log n}$ 
for any $n\times n$ matrix $A$~\cite{Wilkinson_61_01}. Since $\rho(A)$ grows slowly with $n$, the direct algorithm is regarded 
as backwards stable.  However, for the iterative variant pivoting mistakes are
allowed and one might expect that this causes the intermediate growth factors to rapidly grow.  In Theorem~\ref{thm:GEmistakes} 
we show that this does not happen. 

If $A_{i_kj_k}^{(k-1)}$ is the pivot for step $1\leq k\leq \min(m,n)$, then we can quantify the 
quality of this pivot by using the value 
\[
\beta_k  = \left| A_{i_kj_k}^{(k-1)} \right|\bigg/\!\! \left( \max_{1\leq i\leq m}\max_{1\leq j\leq n} \left|A_{ij}^{(k-1)}\right|\right), \qquad \left\|A_{ij}^{(k-1)}\right\|\neq 0,
\]
where $\|\cdot\|$ denotes a matrix norm. 
When $\beta_k =1$ the $k$th pivot is the absolute maximum entry of $A^{(k-1)}$ and no mistake was made. 
When $\beta_k<1$ a pivoting mistake occurred and the severity of the mistake is inversely proportional to 
$\beta_k$. Note that $\beta_k = 0$ is impossible since then the pivoting entry must be nonzero for the GE 
step to be defined.

\section{Gaussian elimination with complete pivoting and mistakes}\label{sec:GEmistakes}
Let $A$ be a nonzero $m\times n$ matrix and $A^{(0)}=A$.  
Suppose that $k$ steps of the iterative variant 
of GE have been performed on $A$ with pivot qualities $\beta_1,\ldots,\beta_k$. Label 
the intermediate matrices during the GE procedure by $A^{(1)},\ldots,A^{(k)}$, 
see~\eqref{eq:GEFirstStep} and~\eqref{eq:GESecondStep} for $A^{(1)}$ and $A^{(2)}$. We 
can bound the intermediate growth factors as follows.

\begin{theorem} 
Let $A$ be a nonzero $m\times n$ matrix and suppose that $1\leq k\leq \min(m,n)$ GE steps are performed on $A$ with 
pivot quality $0<\beta_1,\ldots,\beta_k\leq 1$. Then, the intermediate growth factor after $k$ GE steps on $A$ is
bounded by 
\[
\rho_k(A) \leq 2(\beta_k+\beta_k^{-1})\left(\beta_1^2\prod_{r=2}^{k-1}\beta_{r}^{\tfrac{1}{k-r}}\right)^{-1} \sqrt{k} k^{\tfrac{1}{4}\log k},\qquad 1\leq k\leq \min(m,n).
\]
\label{thm:GEmistakes} 
\end{theorem} 
\begin{proof} 
Fix $1\leq k\leq \min(m,n)$ and suppose that GE selects pivots at the entries $(i_1,j_1),\ldots,(i_k,i_k)$. 
Take the corresponding $k\times k$ submatrix of pivoting entries given by 
\[
B_{st} = A_{i_sj_t}, \qquad 1\leq s,t\leq k 
\]
and, for $0\leq r\leq k-1$, let $B^{(r)}$ be the reduced matrices defined by 
\[
(B^{(r)})_{st} = A_{i_sj_t}^{(r)}, \qquad r+1\leq s,t\leq k. 
\] 
We now bound the intermediate growth factors of $B$, before relating the bound to the intermediate growth factors of $A$.  The proof is a generalization of an argument used by Wilkinson~\cite{Wilkinson_61_01}. 

If $p_1,\ldots,p_k$ denote the absolute values of the pivoting entries, i.e., $p_r = |B^{(r)}_{rr}|$, then  
\begin{equation}
 \left|\det( B^{(r)} )\right| = \prod_{i=r+1}^k p_i,\qquad 0\leq r \leq k-1.
\label{eq:4.2}
\end{equation} 
where we used the fact that the determinant of a matrix is equal to plus or minus the product of the GE pivots.

On the other hand, by Hadamard's inequality on the determinant we know that  
\begin{equation} 
 \left|\det( B^{(r)} )\right| \leq ( (k-r)^{1/2} \beta_{r+1}^{-1} p_{r+1}  )^{k-r},\qquad 0\leq r\leq k-1,
\label{eq:4.3}
\end{equation} 
where we used the fact that each entry of $B^{(r)}$ is bounded above by $p_{r+1}/\beta_{r+1}$ and therefore, the 
2-norm of a column of $B^{(r)}$ is bounded above by $(k-r)^{1/2}\beta_{r+1}^{-1}p_{r+1}$.

Combining~\eqref{eq:4.2} and~\eqref{eq:4.3} we obtain the inequality
\begin{equation} 
 \prod_{i=r+1}^k p_i \leq ( (k-r)^{1/2} \beta_{r+1}^{-1} p_{r+1}  )^{k-r},\qquad 0\leq r\leq k-1.
\label{eq:4.4}
\end{equation} 
Let $\log p_r = q_r$ and take the logarithm of~\eqref{eq:4.4} and also of~\eqref{eq:4.2} with $r = 0$, to obtain 
\begin{equation}
 \sum_{i=r+2}^k q_i \leq \frac{k-r}{2}\log(k-r) - (k-r)\log(\beta_{r+1}) + (k-r-1)q_{r+1},\qquad 1\leq r\leq k-2,
 \label{eq:4.7}
\end{equation} 
and
\begin{equation} 
 \sum_{i=1}^k q_i = \log\left|\det B\right|.
\label{eq:4.8}
\end{equation} 
We now divide the equation in~\eqref{eq:4.7} by $(k-r)(k-r-1)$ for $1\leq r\leq k-2$ 
and divide~\eqref{eq:4.8} by $k-1$, before adding them together. By observing that 
\[
 \frac{1}{r(r-1)}+\frac{1}{(r+1)r} + \cdots + \frac{1}{(k-1)(k-2)} + \frac{1}{k-1} = \frac{1}{r-1}, 
\]
we obtain the following inequality
\[
 \sum_{r=1}^{k-1}\frac{q_{r+1}}{k-r}+\frac{q_1}{k-1}\leq \frac{1}{2}\log\left(\prod_{r=2}^{k-1}r^{\tfrac{1}{r-1}}\right) + \frac{1}{k-1}\log\left|\det B\right| - \log\left(\prod_{r=1}^{k-2}\beta_{r+1}^{\tfrac{1}{k-r-1}}\right) + \sum_{r=1}^{k-2}\frac{q_{r+1}}{k-r}. 
\]
Defining $f(s) = \left(\prod_{r=2}^{s}r^{\tfrac{1}{r-1}}\right)^{1/2}$ and canceling terms, we arrive at
\[
 q_k + \frac{q_1}{k-1} \leq \log f( k-1 ) + \frac{1}{k-1}\log\left|\det B\right| - \log\left(\prod_{r=1}^{k-2}\beta_{r+1}^{\tfrac{1}{k-r-1}}\right). 
\]
Using~\eqref{eq:4.3} for $\left|\det B\right|$ we conclude that 
\[
 q_k + \frac{q_1}{k-1} \leq \log f( k-1 ) + \frac{k}{2(k-1)}\log k -\frac{k}{k-1}\log\beta_1 + \frac{kq_1}{k-1} - \log\left(\prod_{r=1}^{k-2}\beta_{r+1}^{\tfrac{1}{k-r-1}}\right).
\]
In other words, $q_k-q_1$ can be bounded above by 
\[
 q_k - q_1 \leq \log f( k ) + \frac{1}{2}\log k - 2\log\beta_1 - \log\left(\prod_{r=1}^{k-2}\beta_{r+1}^{\tfrac{1}{k-r-1}}\right).
\]
Finally, using the relation $q_k = \log p_k$ and the fact that $p_k/p_1 \geq \beta_{k}^{-1}\rho_{k-1}(A)$ we have 
\[
\beta_{k}^{-1}\rho_{k-1}(A)\leq \frac{p_k}{p_1}\leq \left(\beta_1^2\prod_{r=2}^{k-1}\beta_{r}^{\tfrac{1}{k-r}}\right)^{-1} \sqrt{k} f(k).
\] 
To obtain a bound on $\rho_k(A)$ we note that $\rho_k(A)\leq (1+\beta_k^{-2}) \rho_{k-1}(A)$ since the $k$th pivot has 
quality $\beta_k$. Therefore, we have
\[
 \rho_k(A) \leq (\beta_k+\beta_k^{-1})\left(\beta_1^2\prod_{r=2}^{k-1}\beta_{r}^{\tfrac{1}{k-r}}\right)^{-1} \sqrt{k} f(k),
\]
and the result follows by noting that $f(k) \leq 2k^{\tfrac{1}{4}\log(k)}$~\cite[eq.~(4)]{Foster_97_01}. 
\end{proof} 

The proof of Theorem~\ref{thm:GEmistakes} closely follows Wilkinson's original analysis of the growth factor for GE 
with complete pivoting~\cite{Wilkinson_61_01}. There are four generalizations: 
(1) The analysis allows for pivoting mistakes, (2) The intermediate 
growth factors are bounded, (3) The analysis allows for rectangular matrices, 
and (4) The proof does not require that $m$ and $n$ are finite, though $k$ must be, allowing the 
analysis to also apply to the approximation of functions.\footnote{Theorem~\ref{thm:GEmistakes} shows that when GE
is applied to a function of two variables~\cite{Geddes_08_01,Townsend_13_02,Townsend_13_01}, then it is reasonable 
to select a pivot as a large function value, as opposed to the absolute maximum of the function.  For example, one can select the pivot as the 
largest value on a sufficiently dense sampling grid, as employed in~\cite{Townsend_13_01}.}  If $m=n$ is finite, $k = n$, and $\beta_r = 1$ for $1\leq r\leq n$, then 
the growth factor bound in~\cite{Wilkinson_61_01} is recovered, up to a factor of $2$. 
 
The bound on $\rho_k(A)$ in Theorem~\ref{thm:GEmistakes} reveals a correcting  
phenomenon. Naively, one would expect that it is possible for pivoting mistakes to 
accumulate in a multiplicative fashion; however, Theorem~\ref{thm:GEmistakes}
reveals that this is not possible. For example, the bound on $\rho_k(A)$ 
depends on the quality of the second GE pivot by the factor $\beta_2^{-1/(k-2)}$, and 
as $k$ increases the influence of $\beta_2$ on the intermediate growth factor diminishes. 
It is as if GE can use later steps to slowly correct earlier pivoting mistakes. 

Figure~\ref{fig:GEselfCorrects2} shows the intermediate growth factors when 
GE with complete pivoting is applied to a randomly 
generated\footnote{The matrix is generated using the MATLAB code: {\tt rng(7); A = randn(100)}.}  $100\times 100$ 
matrix, where pivoting mistakes are not allowed.   
Figure~\ref{fig:GEselfCorrects} shows the intermediate growth factors when 
GE with complete pivoting is applied to the same matrix, where pivoting mistakes 
are allowed every ten GE steps. For the experiments a pivot of quality $0<\beta_k\leq 1$ is selected 
by sorting all the possible candidate entries and picking a pivot that has the smallest quality that is 
$\geq \beta_k$. Ties are broken arbitrarily. Two examples below further illustrate the bound in Theorem~\ref{thm:GEmistakes}. 
\begin{figure} 
\centering
 \begin{overpic}[width=.9\textwidth, height=.25\textwidth]{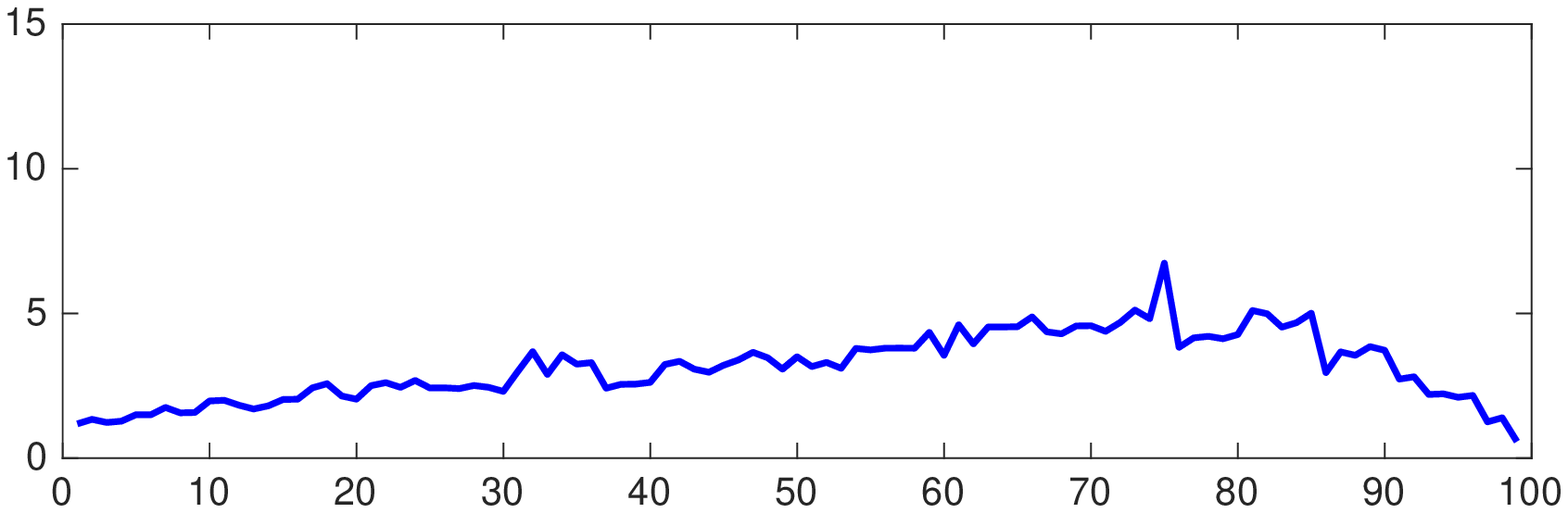}
 \put(20,22) {GE with complete pivoting, no pivoting mistakes}
 \put(2,14) {$\rho_k(A)$}
 \put(94,13) {$\rho(A)$}
 \put(47,-2) {GE step}
 \linethickness{0.5pt}
 \put(11,13.3) {\line(2,0){82}}
 \end{overpic}
 \caption{The intermediate growth factors of GE with complete pivoting when applied to a randomly generated 
 $100\times 100$ matrix.}
\label{fig:GEselfCorrects2}
 \end{figure}
 
\begin{figure} 
\centering
 \begin{overpic}[width=.9\textwidth, height=.25\textwidth]{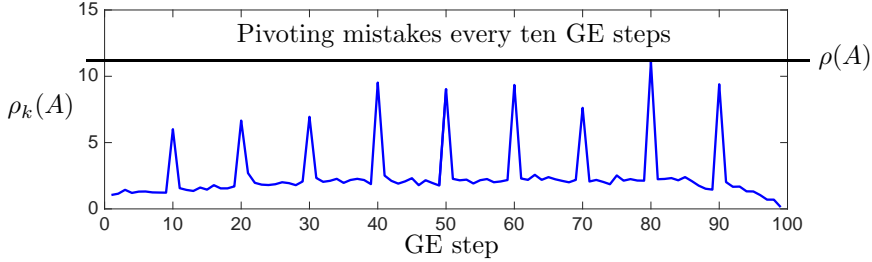}
 \put(28,22) {Pivoting mistakes every ten GE steps}
 \put(2,14) {$\rho_k(A)$}
 \put(94,19.5) {$\rho(A)$}
 \put(47,-2) {GE step}
 \linethickness{0.5pt}
 \put(11,19.9) {\line(2,0){82}}
 \end{overpic}
 \caption{GE with complete pivoting corrects pivoting mistakes. Here, we perform GE with complete pivoting on a randomly generated 
 $100\times 100$ matrix, where every ten steps a pivoting mistake is made, i.e., $\beta_{10k} = 1/10$ for $1\leq k\leq 9$. 
 One can see that the intermediate growth rates quickly 
 recover after a mistake and they do not accumulate in a multiplicative fashion. This phenomenon is important for the success of GE in iterative mode
 as one can safely employ cheaper pivoting strategies that only find large pivots, as opposed to the largest pivot.}
\label{fig:GEselfCorrects}
 \end{figure}

\subsubsection*{Example 1}
Suppose that all the pivots have the same quality $0<\beta\leq 1$, i.e., $\beta_r = \beta$, for $1\leq r\leq k$. 
By Theorem~\ref{thm:GEmistakes}, the intermediate growth rate is bounded above by  
\[
\begin{aligned} 
 \rho_k(A) &\leq 2(\beta+\beta^{-1})\left(\beta\prod_{r=2}^{k-1}\beta^{\tfrac{1}{k-r}}\right)^{-1}\sqrt{k} k^{\tfrac{1}{4}\log k} \cr&\leq 4\left( \beta^{2 + \sum_{r=1}^{k-2} \tfrac{1}{r} } \right)^{-1}\sqrt{k} k^{\tfrac{1}{4}\log k} \cr &= 4\beta^{-2-H_{k-2}}\sqrt{k} k^{\tfrac{1}{4}\log k},
\end{aligned} 
\]
where $H_k$ is the $k$th harmonic number.  Since $H_k\sim \log k$ we have,
\[
  \rho_k(A) \lesssim k^{\tfrac{1}{2}-\log \beta+\tfrac{1}{4}\log k}. 
\]
For example, if $\beta = 1/100$ then the bound on the intermediate growth rate degrades compared 
to complete pivoting without mistakes by an asymptotic factor of only $k^{\log 100} = k^{4.60\ldots}$ as $k\rightarrow \infty$. 
This shows that GE is very robust to severe pivoting mistakes. 

\subsubsection*{Example 2}
To further show that pivoting mistakes are corrected by later GE steps, we take 
$\beta_{r} = \min(1/\log(k-r+2),1)$ for $1\leq r\leq k$. Here, the first few pivots are of 
poor quality, while the quality of later pivots is much better. By Theorem~\ref{thm:GEmistakes}, we have
\[
\begin{aligned}
  \rho_k(A) &\leq 2\left(\beta_1^2\prod_{r=1}^{k-1}\beta_{r}^{\tfrac{1}{k-r}}\right)^{-1}\sqrt{k}k^{\tfrac{1}{4}\log k} \cr
  &=2\log(k+2)^2\prod_{r=1}^{k-1} \log(r+2)^{1/r}\sqrt{k}k^{\tfrac{1}{4}\log k} \cr 
  &\lesssim \log(k)^2 k^{\log\log k}\sqrt{k}k^{\tfrac{1}{4}\log k}. 
\end{aligned}
\]
Therefore, even though the quality of early pivots is logarithmically poor the 
bound on the intermediate growth rate degrades by only a factor of $\log(k)^2 k^{\log\log k}$ as $k\rightarrow\infty$.  This 
example shows that GE is able to correct pivoting mistakes, and could lead to faster algorithms because
pivoting strategies can be far less stringent in the first few steps of GE.   

%
%

\section{Partial pivoting does not correct pivoting mistakes}
%
While GE with complete pivoting is robust to pivoting mistakes, GE with partial pivoting 
is not. 

Consider GE with partial pivoting, where the $k$th pivot is selected as the 
absolute maximum entry in column $k$. Here, we define the quality of the 
$k$th pivot as a value $0<\beta_k\leq 1$ defined by
\[
 \beta_k  = \left| A_{i_kk}^{(k-1)} \right|\bigg/\!\! \left( \max_{1\leq i\leq m} \left|A_{ik}^{(k-1)}\right|\right), \qquad  \max_{1\leq i\leq m} \left|A_{ik}^{(k-1)}\right|\neq 0,\qquad 1\leq k\leq \min(m,n).
\]
Pick $m = n$ and suppose that the quality of the pivots is $0<\beta_1,\ldots,\beta_{n}\leq 1$. 
The following two matrices are obtained by modifying Wilkinson's canonical matrix~\cite{Higham_89_01}:
\[
 A_1 = \begin{pmatrix} 
  1 & & & &1\\[3pt] 
  -\beta_1^{-1} & 1 & & &1 \\[3pt] 
  -\beta_1^{-1} & -\beta_2^{-1} & 1& &1 \\[3pt]
  \vdots & \vdots & \ddots &\ddots& \vdots \\[3pt] 
  -\beta_1^{-1} & -\beta_2^{-1} & \ldots &-\beta_{n-1}^{-1} & 1\\[3pt]
 \end{pmatrix}, \quad A_2 = \begin{pmatrix} 
  1 & & & &1\\[3pt] 
  -\gamma_1^{-1} & 1 & & &1 \\[3pt] 
  -\gamma_1^{-1} & -\gamma_2^{-1} & 1& &1 \\[3pt]
  \vdots & \vdots & \ddots &\ddots& \vdots \\[3pt] 
  -\gamma_1^{-1} & -\gamma_2^{-1} & \ldots &-\gamma_{n-1}^{-1} & 1\\[3pt]
 \end{pmatrix}, 
\]
where 
\[
\gamma_k^{-1} = \begin{cases} \beta_k^{-1}, & \beta_k\neq 1,\cr 0, & \beta_k = 1.\end{cases}
\]
Both $A_1$ and $A_2$ show that pivoting mistakes can accumulate for GE with partial pivoting. While $A_1$ 
gives the worst possible intermediate growth factors, the accumulation of pivoting mistakes is more obvious when 
considering $A_2$ (see Figure~\ref{fig:GEPP}).   For $A_1$ and $A_2$ the intermediate growth factors are given by 
\[
 \rho_k(A_1) = \prod_{r=1}^k \left( 1 + \beta_r^{-1}\right)\Bigg/ \left(\max_{1\leq r\leq n} \beta_r^{-1}\right),\quad 
\rho_k(A_2) = \prod_{r=1}^k \left( 1 + \gamma_r^{-1}\right)\Bigg/ \left(\max_{1\leq r\leq n} \gamma_r^{-1}\right).
\]
Figure~\ref{fig:GEPP} shows the intermediate growth factors when GE with partial pivoting is applied to $A_2$, where 
pivoting mistakes are allowed every ten GE steps. 

\begin{figure} 
\centering
 \begin{overpic}[width=.9\textwidth, height=.25\textwidth]{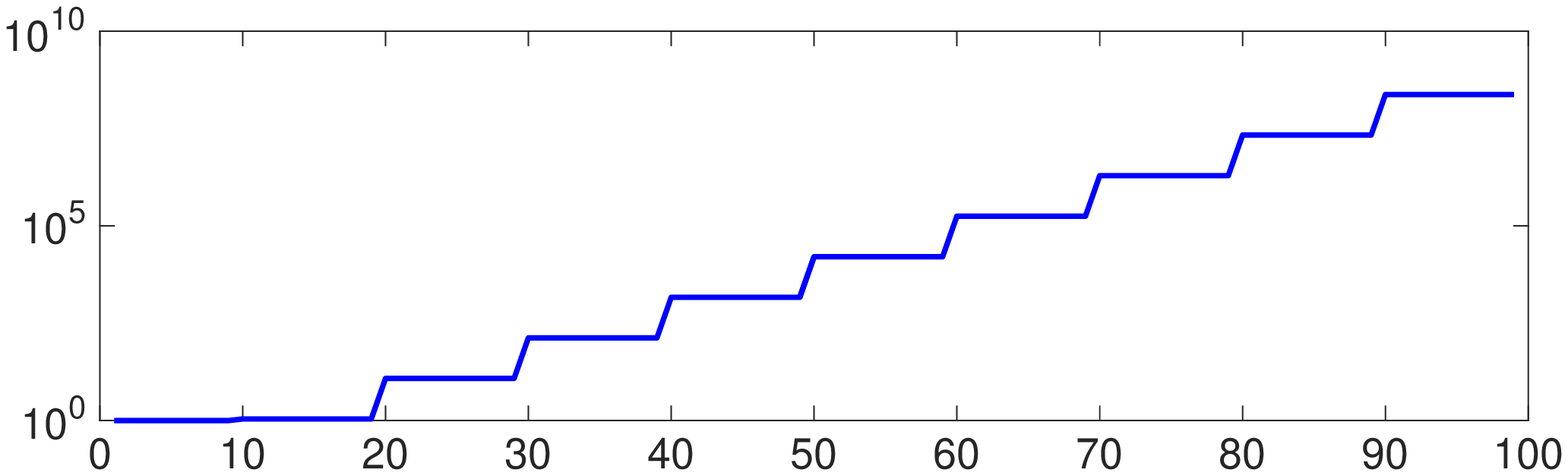}
 \put(28,23) {Pivoting mistakes every ten GE steps}
 \put(0,13) {$\rho_k(A_2)$}
 \put(94,22) {$\rho(A_2)$}
 \put(45,-2) {GE step}
 \linethickness{0.5pt}
 \put(11,22.1) {\line(2,0){82}}
 \end{overpic}
 \caption{GE with partial pivoting does not correct pivoting mistakes. Here, we perform GE with partial pivoting on 
 the matrix $A_2$, where a pivoting mistake is incurred every ten steps, i.e., $\beta_{10k} = 1/10$ for $1\leq k\leq 9$. 
 The pivoting mistakes accumulate in the intermediate growth factors in a multiplicative fashion.}
\label{fig:GEPP}
 \end{figure}
 
GE with partial pivoting is not used in iterative mode because it does not adequately 
construct near-best low rank approximations~\cite{Foster_06_01}. It may be that robustness 
to pivoting mistakes and rank-revealing properties of a pivoting strategy are somehow intimately 
connected.  

\section*{Acknowledgments}
I thank Grady Wright for carefully reading a draft version of the paper, and
Akil Narayan for discussing the topic with me during a productive visit to the 
University of Massachusetts Dartmouth.  I also benefited from comments from 
Mikael Slevinsky and Marcus Webb.

\end{document}